\documentclass[11pt,a4paper]{article}

\usepackage[english]{babel}
\usepackage{texab}
\usepackage{amsmath}
\usepackage{enumerate}              
\usepackage{amsthm}
\usepackage{natbib}
\usepackage{color}
\usepackage{bbm}
\usepackage{setspace}

\newtheorem{definition}{Definition}[section]
\newtheorem{lemma}[definition]{Lemma}
\newtheorem{Auxlemma}[definition]{Auxiliary Lemma}
\newtheorem{theorem}[definition]{Theorem}
\newtheorem{Coro}[definition]{Corollary}
\newtheorem{remark}{Remark}
\newtheorem{condition}{Condition}[section]

\newcommand{\be}{\begin{eqnarray}}
\newcommand{\ee}{\end{eqnarray}}
\newcommand{\bes}{\begin{eqnarray*}}
\newcommand{\ees}{\end{eqnarray*}}
\newcommand*{\argmin}{\mathop{\arg \min}\limits}

\newcommand{\KX}{K_{X}}
\newcommand{\CC}{C}

\begin{document}

\begin{center}
\LARGE{Censored linear model in high dimensions}
\\\vspace{1.5cm}
 \normalsize{\textit{Penalised linear regression on high-dimensional data with left censored response variable}}
\vspace{1cm}

\Large{Patric M\"uller and Sara van de Geer \\Seminar f\"ur Statistik, ETH Zurich}
\end{center}
\vspace{1.5cm}

\begin{abstract}
Censored data are quite common in statistics and have been studied in depth in the last years (for some early references, see \cite{Powell}, \cite{CSR}, \cite{Chay_Cens}). In this paper we consider censored high-dimensional data.\\
High-dimensional models are in some way more complex than their low-dimensional versions, therefore some different techniques are required. For the linear case appropriate estimators based on penalised regression, have been developed in the last years (see for example \cite{bickel2009sal}, \cite{koltch09b}). In particular in sparse contexts the $l_1$-penalised regression (also known as LASSO) (see \cite{Tibshirani96}, \cite{BvdG2011} and reference therein) performs very well.
Only few theoretical work was done in order to analyse censored linear models in a high-dimensional context.\\
We therefore consider a high-dimensional censored linear model, where the response variable is left-censored. We propose a new estimator, which aims to work with high-dimensional linear censored data. Theoretical non-asymptotic oracle inequalities are derived.
\end{abstract}
\newpage
\section{Introduction}
Censored data are quite common in statistics and have been studied in depth in the last years (for some early references, see \cite{Powell}, \cite{CSR}). We consider the censored linear model, where the response variable is left censored and its non-censored version linearly depends on the predictors. Observed are the predictors, the censored response variable and the censoring level. Our main goal is to recover the linear dependency between the covariates and the uncensored response variable. An example fitting this model is the Social Security Administration earnings records in the years '60s are censored at the 'taxable maximum', that is anyone earning more than the maximum is recorded as having earned at the maximum (see \cite{Chay_Cens}). In particular, \cite{Powell} proposed an estimator and proved its strong consistency.\\
In a high-dimensional context, where the dimension of the parameter set $p$ is bigger than the number of observations $n\ll p$ the method of Powell is not directly applicable because it would be underdefined. In order to overcome the difficulties given by the high-dimensional case a new estimator is required.\\
There is a large body of work on linear high-dimensional models. A common approach is to construct penalised estimators like the LASSO (least absolute shrinkage and selection operator proposed in \cite{Tibshirani96}) or the ridge regression and elastic net (see \cite{zou05}). The LASSO is widely studied (see e.g.\ \cite{koltchinskii2011oracle} and  \cite{BvdG2011} and references therein) and gives remarkable results in sparse contexts.\\
In this paper LASSO techniques are combined with the ideas of \cite{Powell} in order to obtain a pertinent estimator for the high-dimensional censored linear model. We prove theoretical results and give oracle bounds for both the prediction and the estimation error of our estimator.
Simulation supporting the theoretical results are also presented.\\
\medskip \\
The paper is organised as follows. We begin in Section \ref{Sec_Model} with the description of the model and the required notation and assumptions for the main theorem, which is given in Section \ref{Sec_Main_Res}. The proof of the theorem and the required technical tools can be found in Section \ref{Sec_Proof}. Finally in Section \ref{Sec_Simul} we present simulations.

\section{Model description and notation}\label{Sec_Model}
Let  $(x,c)\in\R^p\times \R$ be a regression random vector with distribution $Q_{x,c}$ and $\varepsilon \in \R$ be an error term with cumulative distribution function $\nu_0$. Moreover, $(x_i,c_i)$ and $\varepsilon_i,~ i=1,\ldots,n$ are $\iid$ copies of $(x,c)$ and $\varepsilon$ respectively.
\medskip \\
Consider the following left-censored linear model:
\be
y_i=\max\left\lbrace c_i,x_i\beta^0+\varepsilon_i\right\rbrace  ~ \qquad i=1,\ldots,n . \label{modelGen}
\ee
The dependent variable $y_i$, the regression vector $x_i$ and the censoring level $c_i$ are observed for each $i$, while the (conformable) parameter vector $\beta^0\in\R^p$ and the error term $\varepsilon_i$ are unobserved.
\begin{remark}
Real datasets often have a fixed, known censoring level (e.g. the tax example of \cite{Chay_Cens}).
\end{remark}
The following special cases of our model are of special interest:
\medskip \\
The \textbf{constant-censored model} corresponds to the special case of Model \eqref{modelGen}, where $c_i\equiv c_0$ constant and known.
\be
y_i=\max\left\lbrace c_0,x_i\beta^0+\varepsilon_i\right\rbrace  ~, \qquad i=1,\ldots,n . \label{modelResConst}
\ee
Furthermore we define the \textbf{zero-censored model} as
\be
y_i=\max\left\lbrace 0,x_i\beta^0+\varepsilon_i\right\rbrace  ~, \qquad i=1,\ldots,n . \label{modelRes}
\ee
I.e. $c_i$ in model \eqref{modelGen} is fixed and equal 0.

In the low-dimensional case ($p \ll n$) \cite{Powell} showed that in Model \eqref{modelRes}, under some standard assumptions, the estimator 
\be
\hat\beta^{Powell}=\argmin_{\beta}\frac{1}{n}\sum^n_{i=1} \left|y_i-\max\left\lbrace  0,x_i\beta
\right\rbrace\right| \label{Eq_Powell_Est}
\ee
is strongly consistent.
\medskip \\
Consider now the high-dimensional left-censored linear model, i.e. Model \eqref{modelGen} in the high-dimensional case where $p\gg n$. Estimator \eqref{Eq_Powell_Est} leads now to an underspecified system and can not be used in high-dimensional contexts. A new estimator is therefore required. A common approach to high-dimensional linear data is the so-called $l_1$-penalised regression (LASSO).\\
Combining the $l_1$-penalty with the idea of \cite{Powell} we define a new estimator for $\beta^0$ in the high-dimensional context.
\be \label{def_estimator_Gen}
\hat\beta:=\argmin_{\beta \in \mathcal{B}} \left\{ \frac{1}{n}\sum^n_{i=1} \left|y_i-\max\left\lbrace  c_i,x_i\beta \right\rbrace\right| +
\lambda\left\|\beta\right\|_1 \right\} ~,
\ee
where $\mathcal{B}$ is a bounded set.\\
The idea behind this estimator is that the first term controls the prediction error, whether the second term keeps the sparsity under control. The parameter $\lambda$ is a trade-off parameter. An appropriate choice of $\lambda$ results from Theorem \ref{Coro_0_Cens}.

\subsection*{Notation}
Hereafter we list some important notation we use in this paper.
\medskip \\
We denote by $x_{ij}$ the $j$-th component of the predictor vector $x_i$. 
\smallskip \\
Define now for some real function $f$ on $\R^p\times\R$ the loss function $\rho_f$ and the theoretical risk $P$ as:
$$
\rho_f(x,y,c):=|y-f(x,c)|,
$$
$$
P_{\rho_{f}}:=\ERW{\rho_f(x,y,c)}.
$$
The empirical risk is then 
$$
P_{n,\rho_f}:=\frac{1}{n}\sum_{i=1}^n{\rho_f(x_i,y_i,c_i)}.
$$
\\
Furthermore define
\be
f_0(x,c)&:=&\argmin_a \ERW{|y-a|\Big|x,c}  \label{def_f0}
\ee
and for some $\beta\in \R^p$
\bes
f_\beta(x,c)&:=&x^T\beta\vee c~.
\ees
The excess risk for $f_\beta$ is then
\bes
\mathcal{E}(f):=P_{\rho_{f}}-P_{\rho_{f_0}}.
\ees
Denote with $\|\beta\|_1:=\sum_{j=1}^p |\beta_j|$ the $l_1$-norm of the vector $\beta$ and with $S_\beta:=\{j:\beta_j\neq0\}$ its active set, which has cardinality $s_\beta$. 
Then $S_\beta^c=\{ j:\beta_j=0\}$. Let then $\beta_{j,S}:=\beta_j\cdot \mathbbm{1}_{j\in S}$, $j=1,\ldots,p$. Finally define for some function $f$, $\|f\|^2:=\ERW{f_\beta^2(x,c)}$.

\subsection*{Model Assumptions}
We now list some conditions we require in order to prove our results.
\begin{condition}[Design condition]\label{Cond_design1} ~\\
For some constant $\KX$ it holds that
$$ \max_{i,j} |x_{ij} | \le \KX . $$
\end{condition}
A bound on the $X$-values is a somewhat restrictive assumption. 
However we can often approximate an unbounded distribution with its truncated version. Furthermore this condition is quite standard in high-dimensional contexts (see \cite{BvdG2011}).

\begin{condition}[Design condition II]\label{Cond_design2} ~\\
For some constant $K_0$ it holds that
$ x_i(\beta-\beta^0)$ and $f_0(x_i)-f_\beta (x_i)$ takes values in interval $[-K_0,K_0]$ for all $i=i,\ldots,n$, $\forall \beta \in\mathcal{B}$.
\end{condition}
\begin{condition}[Scaling condition]\label{Cond_scaling}
$$
\ERW{x_{ij}^2}=1   \qquad\qquad \text{for all } j=1\ldots p,~i=1,\ldots p.
$$
\end{condition}
This condition can be obtained by rescaling the variables. (Note that we assumed that $x_i$ are i.i.d. for $i=1,\ldots, p$).
\begin{condition}[Solution uniqueness]\label{Cond_sol_uniqueness}~\\
The function $f_0(x,c)$, defined in \eqref{def_f0}, is uniquely defined.
\end{condition}
\begin{condition}[Error assumptions]\label{Cond_Error_assumption}~\\
The distribution function $\nu_0$ of the error term $\varepsilon$, has median $0$ and is everywhere continuously differentiable, with Lipshitz derivative $\dot{\nu_0}$. Furtheremore assume that $\dot{\nu_0}(0)>0$
\end{condition}
\begin{condition}[Censoring condition]\label{Cond_Censoring_condition}~\\
There exists some constant $C_2>0$ such that for all $\beta$ satisfying $||(\beta-\beta^0)_{S_{\beta^0}^c}||_1\leq 3 ||(\beta-\beta^0)_{S_{\beta^0}}||_1$ it holds that:
\bes
\|f_\beta -f_0\|_2^2\geq C_2 \|x^T(\beta-\beta^0)\|_2^2 .
\ees
\end{condition}
This condition only guaranties that there are not, too many censored data. If $(x_i,c_i)$ are $\iid$ and have a symmetric joint distribution then the censoring condition holds for any $\beta,\beta^0$ with constant $C_2=1/4$.
\begin{condition}[Compatibility condition]\label{Cond_Compatibility_condition}~\\
The \textbf{Compatibility condition} is satisfied for the set $S_{\beta^0}$ if for some
$\phi_0>0$ and all $\beta$ satisfying $||(\beta-\beta^0)_{S_{\beta^0}^c}||_1\leq 3 ||(\beta-\beta^0)_{S_{\beta^0}}||_1$ it holds that:
\bes
||(\beta-\beta^0)_{S_{\beta^0}}||_1^2\leq (\beta-\beta^0)^T \ERW{x^T  x}(\beta-\beta^0) \frac{s_{\beta^0}}{\phi_0^2}.
\ees
Here $\phi_0^2$ is the so called compatibility constant introduced by \cite{vandeG07}.
\end{condition}

\section{Results}\label{Sec_Main_Res}
Consider the high-dimensional version of model \eqref{modelGen} ($p\gg n$). Recall the definition of the $l_1$-penalised estimator \eqref{def_estimator_Gen}.
\bes
\hat\beta=\argmin_{\beta \in \mathcal{B}} \left[ \frac{1}{n}\sum^n_{i=1} \left|y_i-\max\left\lbrace  c_i,x_i\beta \right\rbrace\right| +
\lambda\left\|\beta\right\|_1 \right] ~,
\ees

\begin{theorem}\label{Teo_principale_Gen}
Assume Conditions \ref{Cond_design1}-\ref{Cond_Compatibility_condition} and define 
\bes
\lambda(t):=4\KX \sqrt{\frac{2\log(2p)}{n}}+\KX \sqrt{\frac{8t}{n}}.
\ees
Then for $\lambda\geq4\lambda(t)$ and some constant $\CC$ depending only on $K_0$ and $C_2$ (see Conditions \ref{Cond_design2} and \ref{Cond_Compatibility_condition} respectively), with at least probability $1-1/p$, it holds
\be \label{Eq_Teo1_Gen}
\mathcal{E}(f_{\hat\beta}) \leq \lambda^2 \frac{9s_{\beta^0} \CC}{ \phi_0^2}
\ee
and
\be \label{Eq_Teo2_Gen}
||\hat{\beta}-\beta^0||_1\leq\lambda\frac{6\CC s_{\beta^0}}{\phi_0^2}.
\ee
\end{theorem}

\begin{remark}[Asymptotics]~\\
Asymptotically we have
$$
\mathcal{E}(f_{\hat\beta})=O\left( \frac{s_{\beta^0}\log p}{n}\right) 
$$
and
$$
\|\hat{\beta}-\beta^0\|_1=O \left( s_{\beta^0} \sqrt{\frac{\log p}{n}}\right) 
$$
In a sparse context, where 
$$
\frac{s_{\beta^0}\log p}{n} \rightarrow 0
$$
as $n$, $p$, and possibly also $s_{\beta^0}$ tend to infinity, the excess risk converges to $0$. If furthermore $s_{\beta^0}\sqrt{\frac{\log p}{n}} \rightarrow 0$, then also the estimation error converges to $0$.
\end{remark}
\begin{remark}
The bounds for prediction and estimation error given in Theorem \ref{Teo_principale_Gen} do not depend on the distribution of the censoring factor. In fact the censoring level $c$ has a direct influence on the constant $C_2$ in Assumption \ref{Cond_Censoring_condition}. In general higher values for $c_i$ increase the number of censored data. This leads to smaller $C_2$. 
The fact that the censoring level does not directly appear in the theorem should be understood in the sense that the percentage of censored data is important, not the censoring level.
\end{remark}

We now shortly focus on Model \eqref{modelRes}, i.e. the special case where $c_i$ are all fixed and equal 0. In this case the function $f$, as well as the loss function do not any more depend on $c$. Therefore we just write $f(x)$ and $\rho_f(x,y)$ in spite of $f(x,0)$ and $\rho_f(x,y,0)$ respectively.\\
The estimator \eqref{def_estimator_Gen} can be rewritten as
\bes
\hat\beta^{res}:=\argmin_{\beta \in\mathcal{B}} \left\{ \frac{1}{n}\sum^n_{i=1} \left|y_i-\max\left\lbrace  0,x_i\beta \right\rbrace\right| +
\lambda\left\|\beta\right\|_1 \right\} ~,
\ees
This corresponds to the high-dimensional penalised version of the estimator proposed by Powell.

\begin{Coro}\label{Coro_0_Cens}
Assume the same conditions and use the same definitions as in Theorem \ref{Teo_principale_Gen}, then
for $\lambda\geq4\lambda(t)$ and some constant $\CC$, with at least probability $1-1/p$, it holds
\be \label{Eq_Teo1}
\mathcal{E}(f_{\hat\beta}) \leq \lambda^2 \frac{9s_{\beta^0} \CC}{ \phi_0^2}
\ee
and
\be \label{Eq_Teo2}
||\hat{\beta}^{res}-\beta^0||_1\leq\lambda\frac{6\CC s_{\beta^0}}{\phi_0^2}.
\ee
\end{Coro}
The proof directly follows from Theorem \ref{Teo_principale_Gen} taking $c_i\equiv 0$.

\section{Proofs}\label{Sec_Proof}

\subsubsection*{Preliminary remarks and lemmas}

Denote by $\nu(y|x,c)$ the distribution function of the censored random variable $y$ given $x$ and $c$. Then
\bes
\nu(y|x,c)=\left\lbrace 
 \begin{array}{rl}
   0&\textrm{if } y <c\\
   \nu_0(y-x\beta^0)&\textrm{if } y\geq c
 \end{array}
 \right. ~.
\ees
Consequently $\nu(y|x,c)$ is everywhere differentiable, up to $y=c$.
\bigskip \\
Furthermore we show that $f_0(x,c)=f_{\beta^0}(x,c)$.
\begin{proof}
$\ERW{|y-a|\Big|x,c}$ is minimized by $a=\nu^{-1}\left(\frac{1}{2} |x,c\right)$ (the median of $y$ given
$x$ and $c$).
We have:
\bes
f_0(x,c)&=&\\
\median(y|x,c)&=&\median\big(\max\{x\beta^0+\varepsilon,c\}|x,c\big)\\
&=&\max\{x\beta^0+\median(\varepsilon),c\}\\
&=&\max\{x\beta^0,c\} \quad(=x\beta^0\vee c)\\
&=& f_{\beta^0}(x,c) .
\ees
\end{proof}
\begin{remark}
Because $f_0$ is a minimizer of the excess risk,
\bes
\mathcal{E}(f)\geq 0\quad \forall f~.
\ees
\end{remark}

\begin{lemma}\label{Lemma_E>|f-f_0|_Gen}
Assume Conditions \ref{Cond_sol_uniqueness} and \ref{Cond_Error_assumption}, then for all $\beta\in\mathcal{B}$
\bes
\mathcal{E}(f_\beta)\geq C_1^2 ||f_\beta-f_0||^2_2 .
\ees
\end{lemma}
\begin{proof}
For any $a\geq c$ we have: 
\bes
\ERW{|y-a|\Big|x,c}&=&\ERW{(y-a)\mathbbm{1}\{y> a\}\Big|x,c}-\ERW{(y-a)\mathbbm{1}\{y\leq a\}\Big|x,c}\\
&=&\ERW{y-a|x,c}-2\ERW{(y-a)\mathbbm{1}\{y\leq a\}\Big|x,c}\\
&=&\ERW{y|x,c}-a+2a\nu(a|x,c)-2\ERW{y\mathbbm{1}\{y\leq a\}\Big|x} .
\ees
Thus
\be
&&\ERW{|y-a|\Big|x}-\ERW{|y-f_0(x,c)|\Big|x} \nonumber \\
&=&f_0-a+2a \nu_0(a-x\beta^0)-2f_0 \nu_0(f_0-x\beta^0)\label{Eq_Margin_Cond_1_Gen}\\
&&-2\int_{f_0-x\beta^0}^{a-x\beta^0} (x\beta^0+\varepsilon)d\nu_0(\varepsilon) . \nonumber
\ee
Define $z:=a-f_0$ and first look at the case $f_0>c$, i.e. we have $f_0=x\beta^0$. Then the above expression can be rewritten as:
\be
\eqref{Eq_Margin_Cond_1_Gen}&=&-z+2z\nu_0(z)-2\int_0^z\varepsilon d\nu_0(\varepsilon) \nonumber \\
&=& 2\int_0^z(z-\varepsilon)d\nu_0(\varepsilon) \nonumber \\
&=& 2\int_0^z \big(z-\varepsilon \big) \dot{\nu_0}(0) d\varepsilon +2\int_0^z \big(z-\varepsilon \big) (\dot{\nu_0}(\varepsilon)-\dot{\nu_0}(0))d\varepsilon\label{Eq_Margin_Cond_2_Gen} .
\ee
Using the lipshitz condition in the second integral and integrating we finally obtain:
\be
\eqref{Eq_Margin_Cond_2_Gen}&\geq&z^2\dot{\nu_0}(0)-\frac{L}{3}|z|^3 . \label{Eq_Margin_Cond_3_Gen}
\ee
If $x\beta<c$, then $f_0(x,c)=c$. Define $q:=c-x\beta$. Expression \eqref{Eq_Margin_Cond_1_Gen} can be simplified as follows:
\bes
&&-z+2a\nu_0(z+q)-2c\nu_0(q)-2\int_{q}^{z+q} \big(x\beta^0+\varepsilon \big) d\nu_0(\varepsilon)
\nonumber \\
&=&2\left[\int_0^{z+q} z d\nu_0(\varepsilon) + \int_q^{z+q} q  d\nu_0(\varepsilon)-\int_q^{z+q} \varepsilon  d\nu_0(\varepsilon) \right]\\
&=&2\left[\int_0^z (z-\varepsilon)d\nu_0(\varepsilon) \right. \\
&&\left. + \int_0^{z\wedge q}\varepsilon d\nu_0(\varepsilon) +\int_{z\wedge q}^{z\vee q} z\wedge q d\nu_0(\varepsilon)+\int_{z\wedge q}^{z+ q}(z+q-\varepsilon)d\nu_0(\varepsilon)\right] \\
&\geq&2\int_0^z (z-\varepsilon)d\nu_0(\varepsilon)+0+0+0\\
&\geq& z^2\dot{\nu_0}(0)-L\frac{L}{3}|z|^3~,
\ees
where in the last two steps we used that $q,z\geq0$ and result \eqref{Eq_Margin_Cond_2_Gen}. Resuming, for any $f_0,a\geq c$ we have
\bes
\ERW{|y-a|\Big|x,c}-\ERW{|y-f_0(x,c)|\Big|x,c} \geq (a-f_0)^2 \dot{\nu_0}(0)-\frac{L}{3} |a-f_0|^3 .
\ees
Define now $h_{x,c}(z):=\ERW{|y-a_0+z|\Big|x,c}-\ERW{|y-a_0|\Big|x,c}$, $\Lambda^2:=\dot{\nu_0}(0)$ and $C:=\frac{L}{3}$
Because of the supposed uniqueness of the minimum we have that
\bes
\forall \epsilon>0 ~~\exists\alpha_\epsilon>0 \textrm{ such that }\inf_{\epsilon\leq|z|\leq K_0} h_{x,c}(z)>\alpha_\epsilon .
\ees
Furthermore 
\bes
\forall |z|\leq K_0,\quad h_{x,c}(z)\geq \Lambda^2 z^2-C|z|^3 .
\ees
The function $h_{x,c}(z)$ satisfies all assumptions of Lemma \ref{Lemma_Stadler} (see after). Then applying the lemma we obtain:
\bes
h_{x,c}(z) &\geq& C_1^2 z^2 \qquad \quad \textrm{or equivalently,}\\
\ERW{|y-a|\Big|x,c}-\ERW{|y-f_0(x)|\Big|x,c} &\geq& C_1^2(a-f_0)^2~,
\ees
where $C_1^2$ is defined in the cited lemma. Remark that $C_1$ does not depend on $(x,c)$.\\
Using the iterated expectation we obtain
$$
\ERW{\ERW{|y-f_\beta(x,c)|\Big|x,c}-\ERW{|y-f_0(x,c)|\Big|x,c}} 
$$
$$\geq C_1^2 \ERW{\left( f_\beta(x,c)-f_0(x,c)\right)^2}
$$
$$
\Leftrightarrow \qquad\quad P\rho_{f_\beta} -P\rho_{f_0} \geq C_1^2 ||f_\beta-f_0||^2_2~.
$$
This concludes the proof of Lemma \ref{Lemma_E>|f-f_0|_Gen}.
\end{proof}
\begin{Auxlemma}[\cite{Staedler:10}] \label{Lemma_Stadler}
 Let $h:[-K_0,K_0]\rightarrow [0,\infty[$ have the following properties:
\begin{itemize}
 \item $\forall \epsilon>0~\exists \alpha_\epsilon>0$ such that $\inf_{\epsilon<|z|\leq K_0}h(z)>\alpha_\epsilon$,
 \item $\exists \Lambda>0,~C>0$, such that $\forall |z|\leq K_0,~ h(z)\geq \Lambda^2 z^2-C|z|^3~.$
\end{itemize}
Then $\forall |z|\leq K_0$
\bes
h(z)\geq C_1^2 z^2
\ees
where
\bes
C_1^2:=\min \Big\{\epsilon_0;\frac{\alpha_{\epsilon_0}}{K_0^2}\Big\}~, \qquad \epsilon_0=\frac{\Lambda^2}{2C} ~ .
\ees
\end{Auxlemma}

\begin{lemma}[Concentration inequality]~\\
Define
$$
\gamma(y,c,x):=|y-x\beta\vee c|-|y-x\beta^0\vee c|~,
$$
$$
Z_M:=\sup_{\|\beta-\beta^0\|_1\leq M}\left| \frac{1}{n} \sum_{i=1}^{n} \gamma(y_i,c_i,x_i)-\ERW{\gamma(y_i,c_i,x_i)} \right|
$$
then we have
$$
P\left[ Z_M \geq M\lambda(t) \right]\leq \exp(-t) .
$$
\end{lemma}
\begin{proof}
By Massart's inequality (Theorem 14.2 in \cite{BvdG2011}) we have, for any $t>0$:
$$
P\left[ Z_M \geq \ERW{Z_M}+M\KX \sqrt{\frac{8t}{n}}\right]\leq \exp(-t) .
$$
By Lemma 14.20 in \cite{BvdG2011} (contraction inequality) we have:
$$
\ERW{Z_M}\leq 4M\sqrt{\frac{2\log(2p)}{n}}\cdot \KX .
$$
Consequently, for all $t>0$ and $M>0$
$$
P\left[ Z_M \geq 4M\KX \sqrt{\frac{2\log(2p)}{n}}+M\KX \sqrt{\frac{8t}{n}}\right]\leq \exp(-t)
$$
or
$$
P\left[ Z_M \geq M\lambda(t) \right]\leq \exp(-t) .
$$
\end{proof}

\begin{lemma}[Peeling device]~~\label{Lemma_Peeling}\\
Define for some $\delta \geq 0$
\bes
Z_M^\delta:=\sup_{||\beta-\beta^0||_1\leq M}\frac{\Big| \nu_n(\beta)-\nu_n(\beta^0)\Big|}{||\beta-\beta^0||_1 \vee \delta}~,
\ees
then
\bes
Pr(Z_M^\delta>2\lambda(t))\leq \log_2 \left( \frac{\lceil \log_2 M \rceil}{\lfloor \log\delta \rfloor}\right)  e^{-t} .
\ees
\end{lemma}
\begin{proof}
For the proof we use a Peeling device argument (see \cite{BvdG2011} and \cite{vandeGeer:00}). 
$$
P\left( Z_M^\delta>2\lambda(t)\right) =P\left(\sup_{||\beta-\beta^0||_1\leq M}\frac{\Big| \nu_n(\beta)-\nu_n(\beta^0)\Big|}{||\beta-\beta^0||_1\vee \delta} 
>2\lambda(t)\right)
$$
$$
\leq \sum_{\lfloor j=-\log_2 M\rfloor}^{\lceil -\log\delta-1\rceil} P \left( \sup_{2^{-j-1}\leq ||\beta-\beta^0||_1\leq 2^{-j}} \frac{\Big| 
\nu_n(\beta)-\nu_n(\beta^0)\Big|}{||\beta-\beta^0||_1} >2\lambda(t)\right) 
$$
$$
+P\left(\sup_{||\beta-\beta^0||_\delta\leq M}\frac{\Big| \nu_n(\beta)-\nu_n(\beta^0)\Big|}{\delta}>2\lambda(t)\right)
$$
$$
\leq \sum_{\lfloor j=-\log_2 M\rfloor}^{\lceil -\log\delta-1\rceil} P \left( \sup_{2^{-j-1}\leq ||\beta-\beta^0||_1\leq 2^{-j}} \Big| 
\nu_n(\beta)-\nu_n(\beta^0)\Big| >2^{-j}\lambda(t)\right)+ e^{-t}
$$
$$
\leq \sum_{\lfloor j=-\log_2 M\rfloor}^{\lceil -\log\delta-1\rceil} e^{-t}+ e^{-t}
$$
$$
= (\lceil \log_2 M \rceil-\lfloor \log\delta \rfloor  ) e^{-t} .
$$
\end{proof}

\begin{proof}[Proof of Theorem \ref{Teo_principale_Gen}]~\\
We first show inequality \eqref{Eq_Teo1}.
\be
P_{\rho_{f_{\hat\beta}}}-P_{\rho_{f_0}}&=& -\bigg( P_n-P \bigg)\Big( \rho(f_{\hat\beta})-\rho(f_{\beta^0})\Big)
\label{lines_split_inequalities}\\
&&+P_n\big(\rho(f_{\hat\beta})\big)+\lambda ||\hat\beta||_1-P_n\big(\rho(f_{\beta^0})\big) -\lambda
||\beta^0||_1+ \nonumber\\
&&+\lambda ||\beta^0||_1-\lambda ||\hat\beta||_1 \nonumber
\ee
We now look separately at any line of the right part of the last equation.
\smallskip\\
By definition of $\hat\beta$, the evaluation of the second line is never positive.\\
Recall,
\bes
S(\beta)&=&\{j|\beta_j\neq0\}~,\\
S^c(\beta)&=&\{j|\beta_j=0\}~\textrm{and}\\
s_\beta&=&\#S(\beta)~,
\ees
then 
\bes
||\hat\beta-\beta^0||_1=\sum_{j\in S(\beta^0)}|\hat\beta_j-\beta^0_j|+\sum_{j\in S^c(\beta^0)}|\hat\beta_j|.
\ees
Thus
\bes
\lambda ||\beta^0||_1-\lambda ||\hat\beta||_1&=&\lambda\Big( \sum_{j\in S(\beta^0)}|\beta^0_j|-\sum_{j\in
S(\beta^0)}|\hat\beta_j|-\sum_{j\in S^c(\beta^0)}|\hat\beta_j| \Big)\\
&\leq&\lambda\sum_{j\in S(\beta^0)}|\hat\beta-\beta^0|-\lambda\sum_{j\in S^c(\beta^0)}|\hat\beta_j| .
\ees
It is easy to show that $\|\hat\beta-\beta^0\|_1 \ll n$. Lemma \ref{Lemma_Peeling} leads to 
\bes
\Bigg|-\bigg( P_n-P \bigg)\Big( \rho(f_{\hat\beta})-\rho(f_{\beta^0})\Big)\Bigg|\leq
2\lambda(t)\cdot(||\hat\beta-\beta^0||_1 \vee p^{-2}),
\ees
where the above inequality holds with probability at least $1-1/p$ and is obtained by choosing $t:=2\log p$ and $\delta:=p^{-2}$ in the above cited lemma.
\medskip\\
If $||\hat\beta-\beta^0||_1 \leq p^{-2}$ the estimator is very close to the true parameter. Inequalities \eqref{Eq_Teo1} and \eqref{Eq_Teo2} trivially follows from $p^{-2}\ll\lambda/\phi_0^2$.
\medskip\\
If $||\hat\beta-\beta^0||_1 \geq p^{-2}$ then
\be
\mathcal{E}(f_{\hat\beta})&\leq& (2\lambda(t)+\lambda) \cdot \sum_{j\in S(\beta^0)} |\hat\beta_j-\beta^0_j|
+ (2\lambda(t)-\lambda)\sum_{j\in S^c(\beta^0)} |\hat\beta_j| \label{Eq_Comp_Cond_soddisf}\\
&\leq&(2\lambda(t)+\lambda) \cdot \sum_{j\in S(\beta^0)} |\hat\beta_j-\beta^0_j| \nonumber \\
&\leq&(2\lambda(t)+\lambda)||(\hat\beta-\beta^0)_{S_0}||_1~. \nonumber
\ee
\begin{remark}
From Equation \eqref{Eq_Comp_Cond_soddisf} we obtain the following inequality:
\be
||\hat\beta_{S_0^c}||_1 &\leq& \frac{\lambda+2\lambda(t)}{\lambda-2\lambda(t)}
||(\hat\beta-\beta^0)_{S_0^c}||_1 \nonumber \\
&\leq& 3||(\hat\beta-\beta^0)_{S_0^c}||_1 \label{Eq_Ineq_CCond}
\ee
which allow us to use the compatibility condition.
\end{remark}

Lemma \ref{Lemma_E>|f-f_0|_Gen}, Condition \ref{Cond_Censoring_condition} and the compatibility condition leads to:
\bes
(2\lambda(t)+\lambda)||(\hat\beta-\beta^0)_{S_0}||_1\geq C_1||f_{\hat\beta}-f_{\beta^0}||_2^2 
\\
\geq C_2 C_1||x^T(\hat\beta-\beta^0)||_2^2 \\
\geq \frac{ \phi_0^2}{\CC s_{\beta^0}} ||(\hat\beta-\beta^0)_{S_0}||_1^2~,
\ees
where $\CC:=1/(C_1 C_2)$. Resuming we have 
$$
\frac{ \phi_0^2}{s_{\beta^0} \CC} ||(\hat\beta-\beta^0)_{S_0}||_1^2 \leq(2\lambda(t)+\lambda)||(\hat\beta-\beta^0)_{S_0}||_1.
$$
\eqref{Eq_Ineq_CCond} implies $||\hat\beta-\beta^0||_1 \leq 4||(\hat\beta-\beta^0)_{S_0^c}||_1$, which yields
$$
\frac{ \phi_0^2}{4s_{\beta^0} \CC} ||(\hat\beta-\beta^0)_{S_0}||_1 \leq(2\lambda(t)+\lambda).
$$
So we have 
$$
||(\hat\beta-\beta^0)_{S_0}||_1 \leq \lambda \frac{6s_{\beta^0} \CC}{ \phi_0^2}
$$
and 
$$
\mathcal{E}(f_{\hat\beta})\leq \lambda^2 \frac{9s_{\beta^0} \CC}{ \phi_0^2} .
$$
\end{proof}

\section{Numerical results}\label{Sec_Simul}
In this section we present the result of a simulation study. We compare the following three methods:
\begin{itemize}
\item Our estimator, the censored regression with $l_1$-penalisation (CL).
$$
\hat\beta=\argmin_{\beta\in\mathcal{B}} \left\{ \frac{1}{n}\sum^n_{i=1} \left|y_i-\max\left\lbrace  c_i,x_i\beta \right\rbrace\right| +
\lambda\left\|\beta\right\|_1 \right\},
$$
where $\mathcal{B}$ is a large compact set, (See also \eqref{def_estimator_Gen}.)
\item The non-censored regression with $l_1$-penalty (NL).
$$
\hat\beta^{NL}:=\argmin_{\beta} \left\{ \frac{1}{n}\sum_{i=1}^n \left|y_i-x_i\beta \right| +\lambda\left\|\beta\right\|_1 \right\}.
$$
This corresponds to taking the dataset and directly doing an $l_1$, non-censored regression with $l_1$-penalty on the entire dataset. That is, we consider the censored data as being non-censored data.
\item Restricted dataset non-censored regression with $l_1$-penalty (RL).\\
We take the non-censored data as a new restricted dataset and then do an $l_1$, non-censored regression with $l_1$-penalty on this restricted dataset. We have the following estimator:
$$
\hat\beta^{RL}:=\argmin_{\beta} \left\{ \frac{1}{\#J}\sum_{i\in J} \left|y_i-x_i\beta \right| +\lambda\left\|\beta\right\|_1 \right\} ,
$$
where $J\subseteq\{1,\ldots,n\}$ is the set of indexes corresponding to non-censored values of $Y$ (i.e. $i\in J\Leftrightarrow Y_i>c_i$).
\end{itemize}

\subsection{Designs and settings}
In our study we fit 24 different designs varying the following parameters
\begin{itemize}
 \item [-] $n$, the number of observations. It is chosen to be 20, 30, 50 or 70.
 \item [-] The dimension $p$. This is the number of parameters in our model. $p$ can be $30,~50~,100$ or $250$.
 \item [-] The sparsity $s_{\beta^0}$. This denotes the number of non-zero components of $\beta^0$. The sparsity $s_{\beta^0}$ can be $3$, $5$ or $10$.
 \item [-] The signal to noise ratio (SNR). It is defined as 
$$
SNR:=\sqrt{\frac{\sum_{i=1}^n (x_i\beta^0\vee 0)^2}{\sum_{i=1}^n \varepsilon_i^2}} .
$$
We take SNR to be 2 or 8.
\end{itemize}
The censoring factor $c$ is random normally distributed with mean 0 and standard deviation 2.
\medskip \\
We simulate the dataset as follows: We first randomly generate the $n\times p$ dimensional matrix $X$ (each component of the matrix is an independent realisation of a standard normal distributed random variable). Without loss of generality we take the active set of $\beta^0$ as its first $s_{\beta^0}$ components. These components take value $\pm 1$ with probability $1/2$ each. We take Gaussian 0-mean distributed errors, where the choice of the variance is implicitly given by the SNR. Finally the response variable is then generated by
$$
y_i=\max\{x_i\beta^0+\varepsilon_i,c_i \}
$$
The tuning parameter $\lambda$ is chosen to be $\lambda=8\sqrt{\log p/n}$. This is (around) the theoretical choice given from Theorem \ref{Teo_principale_Gen}.
\medskip \\
In all our simulations about the same percentage of the data is censored. By construction we expect 50\% of the data to be censored.
\bigskip \\
For the three different estimators we compare the prediction error 
$$
\frac{1}{n}\sum_{i=1}^n (x_i\hat\beta-x_i\beta^0)^2
$$
and the estimation error 
$$\|\hat\beta-\beta^0\|_1~.$$
The results of the simulation are summarised in Table \ref{Table_Simul}

\begin{table} \scriptsize \renewcommand{\tabcolsep}{3.5pt} \centering 
\begin{tabular}{||r|l || c|c|c || c|c|c ||}
\hline\hline
\multicolumn{8}{||c||}{Simulation results}\\
\hline\hline
$n^\circ$&setting & \multicolumn{3}{c||}{Estimation error}&\multicolumn{3}{c||}{Prediction
error}\\
\hline
$N$&$n$,$p$,$s_{\beta^0}$,STN&~~~CL~~&~~NL~~&~~RL~~&~~CL~~&~~NL~~&~~RL~~\\
\hline \hline
1&70,250,10,8 & 1.21 & 2.72 & 1.8 & 8.13 & 20.92 & 11.7  \\  
  &          &(0.5)&(0.34)&(0.55)&(3.12)&(1.49)&(2.94)  \\  \hline
2&70,100,10,8 & 0.5 & 2.22 & 0.84 & 3.46 & 20.65 & 5.29  \\  
  &          &(0.18)&(0.33)&(0.56)&(1.39)&(2.49)&(3.18)  \\  \hline
3&70,250,10,2 & 1.81 & 2.99 & 2.35 & 11.26 & 22.24 & 14.26  \\  
  &          &(0.22)&(0.32)&(0.31)&(1.62)&(1.63)&(1.43)  \\  \hline
4&70,100,10,2 & 1.48 & 2.52 & 1.88 & 9.57 & 23.23 & 11.55  \\  
  &          &(0.2)&(0.31)&(0.32)&(1.93)&(2.35)&(2.21)  \\  \hline
5&70,250,5,8 & 0.37 & 2.14 & 0.52 & 1.83 & 13.97 & 2.52  \\  
  &          &(0.19)&(0.29)&(0.48)&(0.92)&(1)&(2.27)  \\  \hline
6&70,100,5,8 & 0.24 & 1.71 & 0.26 & 1.42 & 15.88 & 1.49  \\  
  &          &(0.04)&(0.22)&(0.06)&(0.25)&(2.36)&(0.32)  \\  \hline
7&70,250,5,2 & 0.95 & 2.4 & 1.27 & 4.82 & 15.89 & 6.22  \\  
  &          &(0.15)&(0.24)&(0.3)&(0.94)&(1.38)&(1.52)  \\  \hline
8&70,100,5,2 & 0.87 & 2.02 & 0.98 & 4.91 & 18.27 & 5.25  \\  
  &          &(0.12)&(0.22)&(0.21)&(0.75)&(2.14)&(1.13)  \\  \hline
9&40,100,5,8 & 0.81 & 2.14 & 1.04 & 3.74 & 12.41 & 4.5  \\  
  &          &(0.44)&(0.28)&(0.56)&(2.23)&(1.25)&(2.35)  \\  \hline
10&40,100,5,2 & 1.2 & 2.37 & 1.57 & 5.35 & 14.09 & 6.44  \\  
  &          &(0.21)&(0.32)&(0.42)&(1.17)&(1.78)&(1.52)  \\  \hline
11&40,100,3,8 & 0.25 & 1.38 & 0.28 & 1.15 & 7.79 & 1.26  \\  
  &          &(0.15)&(0.22)&(0.27)&(0.82)&(1.21)&(1.28)  \\  \hline
12&40,100,3,2 & 0.7 & 1.65 & 0.76 & 3.06 & 9.13 & 3.13  \\  
  &          &(0.14)&(0.25)&(0.2)&(0.81)&(1.17)&(0.94)  \\  \hline
13&40,50,5,8 & 0.5 & 2.07 & 0.47 & 2.13 & 11.2 & 1.97  \\  
  &          &(0.35)&(0.39)&(0.33)&(1.7)&(1.77)&(1.51)  \\  \hline
14&40,50,5,2 & 1.07 & 2.01 & 1.24 & 4.68 & 11.25 & 5.37  \\  
  &          &(0.27)&(0.35)&(0.36)&(1.12)&(2.06)&(1.39)  \\  \hline
15&40,50,3,8 & 0.2 & 1.62 & 0.2 & 0.7 & 8.17 & 0.73  \\  
  &          &(0.11)&(0.26)&(0.11)&(0.36)&(1.32)&(0.4)  \\  \hline
16&40,50,3,2 & 0.75 & 1.75 & 0.83 & 2.75 & 8.86 & 3.03  \\  
  &          &(0.18)&(0.28)&(0.21)&(0.75)&(1.41)&(0.76)  \\  \hline
17&40,50,5,8 & 0.45 & 1.91 & 0.62 & 1.93 & 10.45 & 2.62  \\  
  &          &(0.27)&(0.35)&(0.44)&(1.35)&(1.72)&(1.93)  \\  \hline
18&40,50,5,2 & 1.1 & 2.21 & 1.32 & 4.74 & 12.5 & 5.42  \\  
  &          &(0.26)&(0.35)&(0.41)&(1.48)&(1.99)&(1.88)  \\  \hline
19&40,50,3,8 & 0.21 & 1.83 & 0.19 & 0.75 & 9.56 & 0.69  \\  
  &          &(0.11)&(0.37)&(0.07)&(0.36)&(1.88)&(0.27)  \\  \hline
20&40,50,3,2 & 0.66 & 1.6 & 0.74 & 2.5 & 8.47 & 2.93  \\  
  &          &(0.2)&(0.29)&(0.22)&(0.81)&(1.2)&(0.92)  \\  \hline
21&20,30,5,8 & 1.5 & 1.98 & 1.37 & 4.85 & 7.07 & 4.37  \\  
  &          &(0.33)&(0.45)&(0.49)&(0.99)&(1.37)&(1.24)  \\  \hline
22&20,30,5,2 & 1.4 & 1.81 & 1.53 & 4.76 & 6.48 & 4.88  \\  
  &          &(0.33)&(0.45)&(0.42)&(0.96)&(1.03)&(0.97)  \\  \hline
23&20,30,3,8 & 0.58 & 1.28 & 0.67 & 1.75 & 4.36 & 1.96  \\  
  &          &(0.4)&(0.35)&(0.5)&(1.27)&(1.25)&(1.41)  \\  \hline
24&20,30,3,2 & 1.26 & 1.7 & 1.16 & 2.99 & 5.09 & 2.93  \\  
  &          &(0.3)&(0.35)&(0.42)&(0.78)&(1.09)&(1.1)  \\  \hline
\hline
\end{tabular}
\caption{Results of the simulation study. For the 24 different designs ($N=1,\ldots,24$) the performance of our estimator (CL) is compared with (NL) and (RL). Based on 30 replicates per design the average and (in brackets) the standard deviation of the prediction and the estimation errors are given in the table.}
\label{Table_Simul}
\end{table}

\subsection{Conclusion}
First of all we can notice that the NL method works much worse than the other methods in all possible designs for both the prediction and estimation error. Comparing CL with RL one can notice that our estimator works almost in any case (slightly) better than RL. This is not surprising because our estimator also takes into account the censored data.

\begin{remark}[High quantity of censored level]
As suggested in Condition \ref{Cond_Censoring_condition}, taking $c$ constant equal 0 does not affect the quality of the fit. (Simulation results for this statement are not included in the paper because they just almost replicate Table \ref{Table_Simul}).
\end{remark}

\begin{remark}[High quantity of censored data]
Increasing the number of censored components, for example changing the censoring level will clearly reduce the quality of all three estimators. This is not surprising because increasing the number of censored values somehow reduce the information we have in the observations. But as one can expect, our estimator is less affected than the other two estimators by such an increase. This also makes sense, because our estimator is the only one which is somehow able to deal with censored values.
\end{remark}

In conclusion one can summarize the results of the simulation study as follows: 
\begin{itemize}
 \item [-] The simulation seem to confirm the theoretical results.
 \item [-] Do not treat censored data as uncensored, it is better to exclude them from your dataset.
 \item [-] Use an appropriate estimator (like CL) for analysing datasets containing censored data.
 \item [-] The quality of the fit is strongly depending on the percentage of the censored data, but is almost not affected by the distribution of the censoring factor $c$ (since the number of censored data does not change).
\end{itemize}

\bibliography{BibliografiaNew}

\end{document}